\documentclass[11pt]{article}
\pdfoutput=1
\usepackage[UKenglish]{isodate}
\usepackage[T1]{fontenc}
\usepackage[noBBpl]{mathpazo}
\usepackage[dvipsnames]{xcolor}
\usepackage{hyperref}

\usepackage{amsmath, amssymb, amsfonts, amsthm, stmaryrd, tikz, graphicx}
\usetikzlibrary{arrows.meta, shapes, positioning, arrows, automata}

\usepackage[a4paper,left=2.5cm, right=2.5cm, top=2.42cm, bottom=2.42cm]{geometry}
\usepackage{microtype}
\usepackage{titlesec}
\titlespacing*{\section}{0pt}{\baselineskip}{0pt}
\titlespacing*{\subsection}{0pt}{0.5\baselineskip}{0pt}

\numberwithin{equation}{section}
\numberwithin{figure}{section}

\usepackage[shortlabels]{enumitem}
\setlist{leftmargin=0.8cm,topsep=0pt,itemsep=-2pt}
\setlist[enumerate]{label=\rm{(\roman*)}}

\makeatletter
\g@addto@macro\normalsize{%
  \setlength\abovedisplayskip{0.4\baselineskip plus 0.4\baselineskip}
  \setlength\belowdisplayskip{0.4\baselineskip plus 0.4\baselineskip}
  \setlength\abovedisplayshortskip{-0.3\baselineskip}
  \setlength\belowdisplayshortskip{0.4\baselineskip plus 0.4\baselineskip}
}
\makeatother

\renewenvironment{thebibliography}[1]
{ \begin{oldthebibliography}{#1}
  \setlength{\parskip}{0pt}
  \setlength{\itemsep}{2pt plus 0.3ex}
  \bgroup\footnotesize }
{ \egroup \end{oldthebibliography} }

\newtheoremstyle{shdefinition}{\topsep}{0.5\topsep}{}{}{\bfseries}{.}{0.5em}{} 
\newtheoremstyle{shplain}{\topsep}{0.5\topsep}{\itshape}{}{\bfseries}{.}{0.5em}{} 
\theoremstyle{shplain}
\newtheorem{thm}{Theorem}
\newtheorem{cor}[thm]{Corollary}
\newtheorem{theorem}{Theorem}[section]
\newtheorem{lemma}[theorem]{Lemma}
\newtheorem{corollary}[theorem]{Corollary}
\newtheorem{proposition}[theorem]{Proposition}

\theoremstyle{shdefinition}
\newtheorem{que}{Question}
\newtheorem*{notation*}{Notation}
\newtheorem*{acknowledgements}{Acknowledgements}

\newcommand{\<}{\langle}
\renewcommand{\>}{\rangle}
\renewcommand{\leq}{\leqslant}
\renewcommand{\geq}{\geqslant}
\renewcommand{\:}{\colon}
\newcommand{\R}{\mathbb{R}}
\newcommand{\Z}{\mathbb{Z}}
\newcommand{\N}{\mathbb{N}}
\newcommand{\Sone}{\mathbb{S}^1}
\DeclareMathOperator{\support}{Supt}
\newcommand{\supt}[1]{\support(#1)}
\DeclareMathOperator{\fixset}{Fix}
\newcommand{\fix}[1]{\fixset(#1)}

\begin{document}

\begin{center} 
\textbf{\LARGE \boldmath Thompson's group $T$ is $\frac{3}{2}$-generated} \\[11pt]
{\Large Collin Bleak, Scott Harper, Rachel Skipper } \\[22pt]
\end{center}

\begin{center}
\begin{minipage}{0.8\textwidth}
\small Every finite simple group can be generated by two elements and, in fact, every nontrivial element is contained in a generating pair. Groups with this property are said to be $\frac{3}{2}$-generated, and the finite $\frac{3}{2}$-generated groups were recently classified. Turning to infinite groups, in this paper, we prove that the finitely presented simple group $T$ of Thompson is $\frac{3}{2}$-generated. Moreover, we exhibit an element $\zeta\in T$ such that for any nontrivial $\alpha \in T$, there exists $\gamma \in T$ such that $\langle \alpha,\zeta^\gamma\rangle = T$.\par
\end{minipage}
\end{center}

\section{Introduction} \label{s:intro}
By the classification of finite simple groups, every finite simple group is $2$-generated. Other than for the then-yet-undiscovered sporadic groups, this was proved by Steinberg in 1962 \cite{steinberg62}. In that paper, Steinberg asked whether every nontrivial element of a finite simple group is contained in a generating pair, a property known as \emph{$\frac{3}{2}$-generation}. In 2000, Guralnick and Kantor answered this question affirmatively \cite{gk00}. Moreover, completing a long line of research, in 2021, Burness, Guralnick and Harper completely characterised the finite $\frac{3}{2}$-generated groups as the finite groups of which every proper quotient is cyclic \cite{bgh21}. While there are finitely generated infinite simple groups that are not $2$-generated (see \cite{GubaSimple86}), there are no finitely presented simple groups which are known to not be $2$-generated.  In this paper we establish that an important infinite simple group is $\frac{3}{2}$-generated in a strong sense.

In his notes of 1965 \cite{ThompsonHandwritten}, Richard J. Thompson introduces  three infinite groups $F < T < V$ and he shows that these groups are all finitely presented and further that $T$ and $V$ are simple. This made $T$ and $V$ the first known examples of finitely presented infinite simple groups, and for the next three decades all examples of finitely presented infinite simple groups were closely related to these groups of Thompson.  The groups $F$, $T$, and $V$ also have  characterisations (respectively) as particular groups of homeomorphisms of the unit interval, the circle, and Cantor space, and much analysis of these groups takes on a dynamical flavour due to these realisations. Since their introduction, these groups have appeared in a variety of mathematical contexts such as in studying the word problem for groups \cite{thompsonmckenzie, Thompson76}, in homotopy and shape theory \cite{Dydak1,Dydak2,FreydHeller}, in group cohomology \cite{brown3,brown4,browngeoghegan1}, and even in dynamical systems and analysis \cite{GhysSergiescu}. There are now hundreds of research articles strongly connected to these groups, and all three groups have hard questions associated with them. It is a common phenomenon in Thompson group theory that results which are known to hold for $F$ or $V$ resist proof for $T$. A standard expository introduction to the theory of the Thompson groups is Cannon, Floyd and Parry's paper \cite{cfp}. 

Donoven and Harper proved that $V$, along with the related Brin--Thompson groups $nV$ and Higman--Thompson groups $V_n$ and $V_n'$, is $\frac{3}{2}$-generated \cite{dh20}. This gave the first examples of infinite non-cyclic $\frac{3}{2}$-generated groups other than Tarski monsters. The only other examples that have since been found are due to Cox \cite{cox22}. These are the subgroups $\< \mathrm{Alt}(\Z), t \>$ and $\< \mathrm{Alt}(\Z), t^2 \>$ of $\mathrm{Sym}(\Z)$ where $t\:\Z \to \Z$ is the translation $n \mapsto n+1$. (Cox also shows that if $k \geq 3$, then $\< \mathrm{Alt}(\Z), t^k \>$ is a $2$-generated infinite group that is not $\frac{3}{2}$-generated but all of whose proper quotients are cyclic, a necessary condition for $\frac{3}{2}$-generation, which is sufficient for finite groups.)

In this paper, we focus on the group $T$. Identifying the circle $\Sone$ with $\R/\Z$, recall that Thompson's group $T$ is the group of piecewise linear homeomorphisms of $\Sone$ which preserve $\Z[\frac{1}{2}]/\Z$, admit finitely many breakpoints, have all gradients as integral powers of two, and have all breakpoints in $\Z[\frac{1}{2}]/\Z$. While $T$ is known to be $2$-generated (see \cite{FunarKapoudjian,LochakSchneps}), our main result is that $T$ satisfies a well-studied strengthening of the standard $\frac{3}{2}$-generation property.

\begin{thm} \label{t:theorem}
There exists an element $\zeta\in T$ such that for every nontrivial $\alpha \in T$ there exists $\gamma \in T$ such that $\< \alpha, \zeta^\gamma \> = T$.
\end{thm}

\begin{cor} \label{c:theorem}
Thompson's group $T$ is $\frac{3}{2}$-generated.
\end{cor}

Theorem~\ref{t:theorem} not only asserts that each nontrivial element $\alpha \in T$ generates with another element $\beta$ but, moreover, that $\beta$ can always be chosen from the same conjugacy class $\zeta^T$, independent of $\alpha$. This much stronger property (known as having \emph{uniform spread at least one}, see below) is shared with all nonabelian finite $\frac{3}{2}$-generated groups with the exception of $\mathrm{Sym}(6)$ \cite[Theorem~3]{bgh21}.

The element $\zeta$ in Theorem~\ref{t:theorem} can be chosen as the one given by the tree pair in Figure~\ref{fig:zeta}. In addition, as we show in Proposition~\ref{prop:infinite}, if we restrict to elements $\alpha$ of infinite order then we can choose $\zeta$ to be any infinite order element, that is to say, for any two infinite order elements $\alpha, \zeta \in T$ there exists $\gamma \in T$ such that $\< \alpha, \zeta^\gamma \> = T$. This naturally raises the following question.

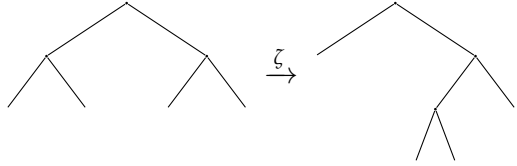
\begin{figure}[t]
  \[
  \begin{tikzpicture}[
      inner sep=0pt,
      baseline=-30pt,
      level distance=20pt,
      level 1/.style={sibling distance=60pt},
      level 2/.style={sibling distance=30pt},
      level 3/.style={sibling distance=15pt}
    ]
    \node (root) [circle,fill] {}
    child {node (0) [circle,fill] {}
      child {node (00) {}}
      child {node (01) {}}}
    child {node (1) [circle,fill] {}
      child {node (10) {}}
      child {node (11) {}}};
  \end{tikzpicture}
  \ \xrightarrow{\,\zeta\,} \
  \begin{tikzpicture}[
      inner sep=0pt,
      baseline=-30pt,
      level distance=20pt,
      level 1/.style={sibling distance=60pt},
      level 2/.style={sibling distance=30pt},
      level 3/.style={sibling distance=15pt}
    ]
    \node (root) [circle,fill] {}
    child {node (0) {}}
    child {node (1) [circle,fill] {}
      child {node (10) [circle,fill] {}
        child {node (100) {}}
        child {node (101) {}}}
      child {node (11) {}}};
  \end{tikzpicture}
\]
{\renewcommand{\thefigure}{1}
\caption{The element $\zeta \in T$} 
\label{fig:zeta}
}
\end{figure}

\begin{que}
Given any element $\tau \in T$ of infinite order and any nontrivial element $\alpha \in T$, does there exist $\gamma \in T$ such that $T = \< \alpha, \tau^{\gamma} \>$?
\end{que}

Even more adventurously, we can ask the following.

\begin{que}
Given any element $\tau\in T$ of order at least three and nontrivial element $\alpha \in T$, does there exist $\gamma \in T$ such that $T=\<\alpha,\tau^{\gamma}\>$?  
\end{que}

The \emph{spread} of a group $G$, written $s(G)$, is the supremum over $k$ such that for any $k$ nontrivial elements $\alpha_1, \dots, \alpha_k \in G$ there exists an element $\beta \in G$ such that $\< \alpha_1, \beta \> = \dots = \< \alpha_k, \beta \> = G$. Therefore, $G$ is $\frac{3}{2}$-generated if and only if $s(G) \geq 1$. It was recently shown that every $\frac{3}{2}$-generated finite group $G$ actually satisfies $s(G) \geq 2$, or said otherwise, no finite group has spread exactly 1 \cite[Corollary~2]{bgh21}. Motivated by this, Donoven and Harper asked whether there exist any infinite groups $G$ with $s(G)=1$ \cite[Question~2]{dh20}. Let us also note that determining the exact spread is difficult in general. For instance, Brenner and Wiegold, in their 1970 paper \cite{bw75} where they introduce spread, prove that $s(\mathrm{Alt}(n)) = 4$ if $n \geq 8$ is even; however, $s(\mathrm{Alt}(n))$ is still unknown for general $n$. This discussion raises our final question.

\begin{que}
What is the spread of $T$? Is it infinite?
\end{que}

Let us also mention the related \emph{uniform spread} of $G$, written $u(G)$, which is the supremum over $k$ for which there exists a conjugacy class $\zeta^G$ such that for any $k$ nontrivial elements $\alpha_1, \dots, \alpha_k \in G$ there exists $\beta \in \zeta^G$ such that $\< \alpha_1, \beta \> = \dots = \< \alpha_k, \beta \> = G$. As noted above, Theorem~\ref{t:theorem} shows that $u(T) \geq 1$, and again it is natural to ask for the precise value of $u(T)$.

We now discuss how we prove our main theorem. There is one common theme to all our proofs but we also approach the two cases where $\alpha$ has infinite and finite order separately, making use of different techniques. The common theme is proving that $\<\alpha, \zeta^\gamma\> = T$ by first covering the circle $\Sone$ by overlapping open intervals and showing that $\<\alpha, \zeta^\gamma\>$ contains the full pointwise stabiliser in $T$ of the complements of each of these intervals (see Lemma~\ref{lem:interval_t}). Such a stabiliser is a topologically conjugated copy of Thompson's group $F$, which we call an $F$-in-the-box.

We prove Theorem~\ref{t:theorem} for infinite order $\alpha$ in Proposition~\ref{prop:infinite}, where we first obtain an approximate $F$-in-the-box (which acts nontrivially elsewhere on the circle) by introducing local versions of standard generators of $F$ via targeted conjugation and then by applying further techniques involving conjugation and commutation to achieve an $F$-in-the-box exactly. We then address the case where $\alpha$ has finite order in Proposition~\ref{prop:finite}. Here we can give an explicit choice of $\alpha$ up to conjugacy, which we then use to write down two elements that we claim generate an $F$-in-the-box. To prove this, we employ a generation criterion of Golan \cite{GolanGeneration} that makes use of the Stalling's $2$-core technology as introduced by Guba and Sapir and later developed by Golan and Sapir \cite{GolanGeneration,GolanSapir,gubaSapirDiagramGroups}; we discuss this criterion in Section~\ref{ss:p_generation}.  To complete the proof in both cases, we use conjugation to find enough copies of our $F$-in-the-box such that the union of their supports covers the circle.

\section{Preliminaries} \label{s:prelims}

This section establishes some groundwork that we will rely on when we prove our main theorem in Section~\ref{s:proofs}. Let us first begin by commenting on our notation, which is mainly standard.

\begin{notation*}
Group actions are on the right, so accordingly $x^y = y^{-1}xy$ and $[x,y] = x^{-1}y^{-1}xy$. We denote the derived subgroup of $G$ by $G'$. If $G$ is a group acting on a set $X$, then for $\gamma \in G$, we write $\fix{\gamma} = \{ x \in X \mid x \gamma = x \}$ and $\supt{\gamma} = X \setminus \fix{\gamma}$. For a function $f\: A \to B$ and a subset $X \subseteq A$, we write $f|_X\: X \to Xf$ for the restriction of $f$ to $X$. We write $\Sone = \R/\Z$ for the unit circle, and for $x \in [0,1)$ we regularly abuse notation by writing $x$ for the coset $x+\Z$; in particular, we will write $ \Z[\frac{1}{2}]$ for the image of the dyadic rationals $\Z[\frac{1}{2}]$ in $\Sone$ as well.
\end{notation*}

\subsection{\boldmath Thompson-like maps and the groups $F$ and $T$} \label{ss:p_interval}

Let $p,q,r,s \in \Z[\frac{1}{2}]$ with $p < q$ and $r < s$. A map $\tau\:(p,q) \to (r,s)$ or $\tau\:[p,q] \to [r,s]$ is \emph{Thompson-like} if $\tau$ is a piecewise-linear homeomorphism with finitely many breakpoints such that all gradients are powers of two and all breakpoints are in $\Z[\frac{1}{2}]$. Let $F_{[r,s]}$ be the group of Thompson-like maps $\tau\:[r,s] \to [r,s]$. \emph{Thompson's group $F$} is defined as $F_{[0,1]}$. In addition, if $[r,s] \subseteq [0,1]$, then we may naturally extend an element of $F_{[r,s]}$ to a map $[0,1] \to [0,1]$ by defining it as the identity on $[0,1] \setminus [r,s]$ and under this identification $F_{[r,s]}$ is the pointwise stabiliser in $F$ of $[0,1] \setminus (r,s)$. We may refer to such a subgroup $F_{[r,s]} \leq F$ as an \emph{$F$-in-the-box} since the graph of any function in $F_{[r,s]}$ acts trivially outside the box $[r,s] \times [r,s] \subseteq [0,1] \times [0,1]$. The following lemma regarding any two $F$-in-the-boxes is well known.

\begin{lemma} \label{lem:interval_f}
Let $p,q,r,s \in \Z[\frac{1}{2}]$ with $p < q$ and $r < s$. Then there exists a Thompson-like map $\tau\:(p,q) \to (r,s)$, and for any such $\tau$ we have $(F_{[p,q]})^\tau = F_{[r,s]}$.
\end{lemma}

It is a standard fact (see \cite{cfp}) that the elements of $F'$ can be characterised as the elements of $F$ that act as the identity in a neighbourhood of $0$ and $1$. In particular, we have the following.

\begin{lemma} \label{lem:simple_f}
Let $r,s \in \Z[\frac{1}{2}]$ with $0 < r < s < 1$. Then $F_{[r,s]} \leq F'$.
\end{lemma}

We now relate the above definitions to the group $T$. For $(r,s) \subsetneq \Sone$, write $T_{[r,s]}$ for the pointwise stabiliser of $\Sone \setminus (r,s)$ in $T$. Just as we discussed above for $F_{[r,s]}$, we may naturally view the elements of $T_{[r,s]}$ as maps supported on $[r,s]$. The next lemma is straightforward.

\begin{lemma} \label{lem:interval_t}
Let $r,s \in \Z[\frac{1}{2}] \cap \Sone$ with $(r,s) \subsetneq \Sone$. Then there exists a Thompson-like map $\tau\:(0,1) \to (r,s)$ such that $T_{[r,s]} = (F_{[0,1]})^\tau$.
\end{lemma}

The following lemmas are folklore results in the literature on Thompson groups and, for convenience, we give straightforward proofs.

\begin{lemma}\label{lem:small_support_f}
Let $a,b,c,d \in \Z[\frac{1}{2}] \cap [0,1]$ with $a < b < c < d$. Then $F_{[a,d]} = \< F_{[a,c]}, F_{[b,d]} \>$.
\end{lemma}

\begin{proof}
Let $\gamma\in F_{[a,d]}$. We will show that $\gamma = \alpha\beta$ for some elements $\alpha \in F_{[a,c]}$ and $\beta \in F_{[b,d]}$. By replacing $\gamma$ by its inverse if necessary, we may assume that $e = b\gamma^{-1}$ satisfies $a < e \leq b < c$. Since $e < c$, we may fix an element $\alpha \in F_{[a,c]}$ such that $\alpha|_{[a,e]} = \gamma|_{[a,e]}$. Since $\alpha|_{[a,e]} = \gamma|_{[a,e]}$, for all $x \in (a,b) = (a,e\gamma) = (a,e\alpha)$ we have $x\alpha^{-1}\gamma = 1$. Therefore, noting that $\alpha^{-1}, \gamma \in F_{[a,d]}$, we have $\beta = \alpha^{-1}\gamma \in F_{[b,d]}$. Now $\gamma = \alpha\beta$ with $\alpha \in F_{[a,c]}$ and $\beta \in F_{[b,d]}$, so $\gamma \in \< F_{[a,c]}, F_{[b,d]} \>$, as claimed.
\end{proof}

\begin{corollary} \label{cor:small_support_f}
Let $k \geq 2$ and let $[r_1,s_1], \dots, [r_k,s_k]$ be dyadic intervals such that $\cup_{i=1}^k(r_i,s_i)=(r,s)$. Then $F_{[r,s]} = \< F_{[r_1,s_1]}, \dots, F_{[r_k,s_k]} \>$.
\end{corollary}

\begin{corollary}\label{cor:repartitioning}
Let $\{(r_i, s_i)\mid 1 \leq i \leq m \}$ and $\{(u_j, v_j) \mid 1\leq j \leq n \}$ be sets of dyadic intervals in $[0,1]$ such that
$
\bigcup_{i=1}^m (r_i, s_i)=\bigcup_{j=1}^n (u_j, v_j).
$
Then $\< F_{[r_i, s_i]} \mid 1\leq i \leq m\>=\< F_{[u_j, v_j]} \mid 1\leq j \leq n\>.$
\end{corollary}

We will call the set $\{(u_j, v_j) \mid 1\leq j \leq n \}$ from Corollary~\ref{cor:repartitioning} a \emph{repartitioning} of $\bigcup_{i=1}^m (r_i, s_i)$ as our usage below is similar in spirit to how partitions of unity are employed.

\begin{lemma} \label{lem:small_support}
Let $k \geq 2$ and let $[r_1,s_1], \dots, [r_k,s_k]$ be dyadic intervals such that $\cup_{i=1}^k(r_i,s_i)=\Sone$. Then $T = \< T_{[r_1,s_1]}, \dots, T_{[r_k,s_k]} \>$.
\end{lemma}

\begin{proof}
Let $\gamma \in T$ be nontrivial and write $G = \< T_{[r_1,s_1]}, \dots, T_{[r_k,s_k]} \>$.  We will show that $\gamma\in G$. First assume that $\fix{\gamma}=\emptyset$. 

Since $\cup_{i=1}^k(r_i,s_i)=\Sone$, we claim that for all $a \in \Z[\frac{1}{2}]$ such that $|a|$ is sufficiently small, $G$ contains the rotation $\rho_a\:\Sone \to \Sone$ defined as $x \mapsto x+a$. To see this, let $a \in \Z[\frac{1}{2}]$ be small enough such that the graph of $\rho_a$ is fully contained in the union of rectangles $\cup_{i=1}^k[r_i,s_i]^2$ within $\Sone\times\Sone$. Now let $\tau\in T_{[r_1,s_1]}$ be an element with inverse $\tau^{-1}$ having graph over $[r_1,s_1]$ agreeing with the graph of $\rho_a$, except in an $\epsilon$-neighbourhood of the ends $r_1$ and $s_1$. By choosing $\epsilon>0$ small enough, the support of $\phi = \rho_a\tau$ is fully contained in $\cup_{i>1}^k(r_i,s_i)$. An inductive application of Lemma~\ref{lem:small_support_f} now establishes that $\phi$ can be written as an element of $\langle T_{[r_2,s_2]},\ldots,T_{[r_k,s_k]}\rangle$, and the claim holds. It follows that for any $b\in\Z[\frac{1}{2}]$ with $0\leq b<1$, the group $G$ contains the rotation $\rho_b$.

Thus, we may multiply $\gamma$ by some rotation $\rho_b\in G$ so as to obtain an element $\gamma'$ such that $\fix{\gamma'} \neq \emptyset$. If $\gamma' = 1$, then we have written $\gamma$ as a product of elements of $G$, so $\gamma \in G$. Otherwise, since $\gamma' \in G$ if and only if $\gamma \in G$, it is sufficient to show that $\gamma' \in G$. Therefore, we now assume that $\gamma\neq 1$ and $\fix{\gamma} \neq \emptyset$. 

Let $p \in \fix{\gamma}$. Since $p\gamma = p$, we can fix $a,b \in \Sone$ such that $a < p < b < a$ in the circular order and such that there exist $c_1,d_1 \in \Z[\frac{1}{2}] \cap \Sone$ satisfying $(a,b) \cup (a\gamma, b\gamma) \subsetneq (c_1,d_1) \subsetneq \Sone$. Since $p \in (a,b) \cap (a\gamma,b\gamma)$, there exist $c_2,d_2 \in \Z[\frac{1}{2}] \cap \Sone$ satisfying $\emptyset \subsetneq (c_2,d_2) \subsetneq (a,b) \cap (a\gamma,b\gamma)$. By Corollary~\ref{cor:repartitioning} (by splitting each $T_{[r_i,s_i]}$ as products of smaller support copies of $F$ through repartitioning, and then merging some of these in some order), we may assume that $k=2$ and that $[r_1,s_1] = [c_1,d_1]$ and $[r_2,s_2] = [d_2,c_2]$ (the complement of $(c_2,d_2)$). Since $[a,b]\cup [a\gamma,b\gamma] \subseteq (r_1,s_1)$, we can fix $\alpha \in T_{[r_1,s_1]}$ such that $\alpha|_{[a,b]} = \gamma|_{[a,b]}$. Therefore, $\beta = \alpha^{-1}\gamma$ acts trivially on $[a\gamma,b\gamma]$, so $\supt{\beta} \subseteq \Sone \setminus [\alpha\gamma,\beta\gamma] \subseteq [r_2,s_2]$. Now $\gamma = \alpha\beta$ with $\alpha \in T_{[r_1,s_1]}$ and $\beta \in T_{[r_2,s_2]}$, so $\gamma \in G = \< T_{[r_1,s_1]}, T_{[r_2,s_2]} \>$ as required.
\end{proof}

\subsection{\boldmath Representing elements of $F$ and $T$} \label{ss:p_notation}

This section explains our notation for elements of $F$ and $T$. For any $\gamma \in T$ (or of $F$), we may fix two sequences of disjoint open dyadic intervals $(I_1, \dots, I_k)$ and $(J_1, \dots, J_k)$ whose closures both cover $\Sone$ (or both cover $[0,1]$, respectively) such that $I_i\gamma = J_i$ and $\gamma|_{I_i}$ is affine for each index $i$. Therefore, these sequences of intervals fully determine $\gamma$, so we write $\gamma = (I_1, \dots, I_k) \mapsto (J_1, \dots, J_k)$. Moreover, by subdividing some of these intervals if necessary, we may assume that each such interval is of the form $(\frac{a}{2^j},\frac{a+1}{2^j})$. This is useful since such dyadic intervals can be naturally labelled by binary words as follows: inductively, the empty word corresponds to $(0,1)$ and for any binary word $v$, the words $v0$ and $v1$ correspond to the left and right halves, respectively, of the interval corresponding to $v$. Therefore, we can specify an element $\gamma \in T$ as $\gamma = (u_1, \dots, u_n) \mapsto (v_1, \dots v_n)$ where $u_i$ and $v_i$ binary words. Moreover, if $\gamma \in F$, then we may assume that, in the lexicographic order, $u_1 < \dots < u_n$ and $v_1 < \dots < v_n$, so it is enough to specify the sets $\{u_1, \dots, u_n\}$ and $\{v_1, \dots, v_n\}$. 

Therefore, by identifying binary words with leaves of a finite binary rooted tree, we can represent an element of $F$ by a \emph{tree pair}, namely a pair of finite binary rooted trees with the same number of leaves. (Elements of $T$ also have a description via tree pairs but we will not need this.) For example, the element $\zeta \in F$ defined as
\[
x\zeta = \left\{ 
\begin{array}{ll}
2x                       & \text{if $x \in [0,\frac{1}{4}]$} \\[3pt]
\frac{1}{2}x+\frac{3}{8} & \text{if $x \in (\frac{1}{4},\frac{3}{4}]$}\\[3pt]
x                        & \text{if $x \in (\frac{3}{4},1]$.}
\end{array}
\right.
\]
is denoted by $(00, 01, 10, 11) \mapsto (0, 100, 101, 11)$ and also by the tree pair in Figure~\ref{fig:zeta}. A tree pair $A \to B$ is \emph{reduced} if there do not exist binary words $u$ and $v$ such that both $u0$ and $u1$ are leaves of $A$ and both $v0$ and $v1$ are leaves of $B$. For example, the tree pair in Figure~\ref{fig:zeta} is reduced.

\subsection{\boldmath A conjugacy result for $T$} \label{ss:p_conjugacy}

Fix an integer $k \geq 2$. We say that $\mu \in T$ \emph{admits $k$ hops (at $p$)} if there exists a point $p \in \supt{\mu}$ such that $(p\mu^i,p\mu^{i+1}) \cap (p\mu^j,p\mu^{j+1}) = \emptyset$ for all distinct $0 \leq i, j < k$. If $\mu$ admits $k$ hops at $p$, then for $0 \leq i < k$ we call $[p\mu^i,p\mu^{i+1})$ a \emph{fundamental domain} of $\mu$. We use the following in Proposition~\ref{prop:infinite}.

\begin{lemma} \label{lem:conjugacy}
Let $k \geq 1$. Let $\mu,\nu \in T$ be elements that admit $k+3$ hops at $p,q \in \Z[\frac{1}{2}] \cap \Sone$ respectively. Let $[r,s] \subsetneq \Sone \setminus [q,q\nu^{k+1}]$. Then there exists an element $\gamma \in T$ such that $\mu^\gamma|_{[q,q\nu^k]} = \nu|_{[q,q\nu^k]}$ and $q\nu^{k+1} < r < s < r\mu^\gamma < s\mu^\gamma < q < q\nu^{k+1}$ in the circular order.
\end{lemma}

\begin{proof}
We begin by defining $\gamma$ inductively on $[p\mu^i,p\mu^{i+1}]$ for $0 \leq i \leq k$. On $[p,p\mu]$, define $\gamma$ in any way such that $[p,p\mu]\gamma = [q,q\nu]$. For $i > 0$, define $\gamma|_{[p\mu^i,p\mu^{i+1}]}$ as $\mu^{-1}|_{[p\mu^i,p\mu^{i+1}]}\gamma|_{[p\mu^{i-1},p\mu^i]}\nu$. This implies that for all $0 \leq i < k$, we have $[p\mu^i,p\mu^{i+1}]\gamma = [q\nu^i,q\nu^{i+1}]$ and $\mu^\gamma|_{[q\nu^i,q\nu^{i+1}]} = \nu|_{[q\nu^i,q\nu^{i+1}]}$ (see Figure~\ref{fig:conjugacy}).

With $\gamma$ defined on $[p,p\mu^{k+1}]$, now turn to $\Sone \setminus [p,p\mu^{k+1}]$. Let $m_1 = p\mu^{k+1} + \frac{1}{2}(p\mu^{k+2}-p\mu^{k+1})$ and $m_2 = p\mu^{k+1} + \frac{3}{4}(p\mu^{k+2}-p\mu^{k+1})$, and let $n_1 = s + \frac{1}{2}(q-s)$ and $n_2 = s + \frac{3}{4}(q-s)$. Define $\gamma$ on $\Sone \setminus [p,p\mu^{k+1}]$ in any way such that $(m_1,m_2,m_1\mu,m_2\mu)\gamma = (r,s,n_1,n_2)$ (see Figure~\ref{fig:conjugacy}). This implies that $[r\mu^\gamma,s\mu^\gamma] = [m_1\mu\gamma,m_2\mu\gamma] = [n_1,n_2]$, so $q\nu^{k+1} < r \leq s < r\mu^\gamma \leq s\mu^\gamma < q < q\nu^{k+1}$ in the circular order, as desired.
\end{proof}
\vspace{-11pt}

\begin{figure}[h]
\centering
\begin{tikzpicture}
[scale=0.55]
\foreach \i [count=\j from 0] in {0,1,2,3,k-1,k} {
\draw (2*\j,6) node (1\i) {
\ifthenelse{\equal{\i}{0}}{$q$}{
\ifthenelse{\equal{\i}{1}}{$q\nu$}{
\ifthenelse{\equal{\i}{3}}{$\dots$}{$q\nu^{\i}$}}}};
\draw (2*\j+1,4) node (2\i) {
\ifthenelse{\equal{\i}{0}}{$p$}{
\ifthenelse{\equal{\i}{1}}{$p\mu$}{
\ifthenelse{\equal{\i}{3}}{$\dots$}{$p\mu^{\i}$}}}};
\ifthenelse{\equal{\i}{3}}{}{
\draw[->] (1\i) -- (2\i);} 
}
\foreach \i [count=\j from 1] in {1,2,3,4,k,k+1} {
\draw (2*\j+1,2) node (3\i) {
\ifthenelse{\equal{\i}{1}}{$p\mu$}{
\ifthenelse{\equal{\i}{4}}{$\dots$}{$p\mu^{\i}$}}};
\draw (2*\j,0) node (4\i) {
\ifthenelse{\equal{\i}{1}}{$q\nu$}{
\ifthenelse{\equal{\i}{4}}{$\dots$}{$q\nu^{\i}$}}};
\ifthenelse{\equal{\i}{4}}{}{
\draw[->] (3\i) -- (4\i)};
}
\draw[->] (20) -- (31);
\draw[->] (21) -- (32);
\draw[->] (22) -- (33);
\draw[->] (2k-1) -- (3k);
\draw[->] (2k) -- (3k+1);
\draw (15,2) node (37) {$m_1$};
\draw (17,2) node (38) {$m_2$};
\draw (19,2) node (39) {$p\mu^{k+2}$};
\draw (21,2) node (310) {$m_1\mu$};
\draw (23,2) node (311) {$m_2\mu$};
\draw (25,2) node (312) {$p\mu^{k+3}$};
\draw (27,2) node (313) {$p$};
\draw (14,0) node (47) {$r$};
\draw (16,0) node (48) {$s$};
\draw (20,0) node (410) {$n_1$};
\draw (22,0) node (411) {$n_2$};
\draw (26,0) node (413) {$q$};
\foreach \i in {7,8,10,11,13} {
\draw[->] (3\i) -- (4\i);
}
\draw (-0.5,5) node {$\gamma^{-1}$};
\draw (1.25,2.75) node {$\mu$};
\draw (1.75,1.25) node {$\gamma$};
\end{tikzpicture}
\caption{An accompaniment to the proof of Lemma~\ref{lem:conjugacy}}
\label{fig:conjugacy}
\end{figure}
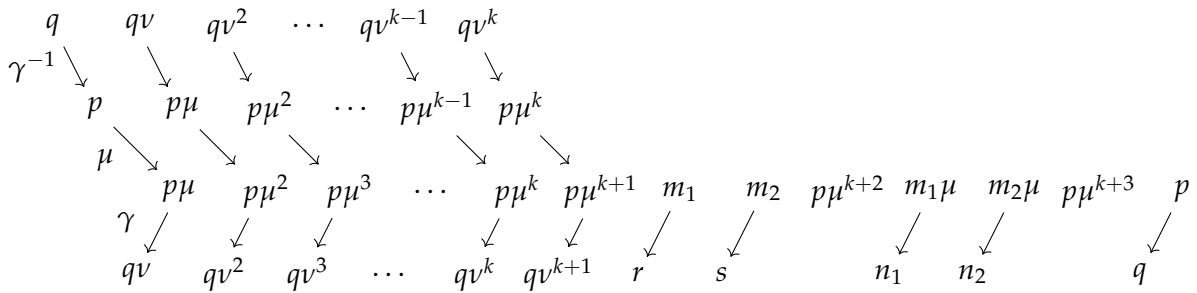

\subsection{\boldmath A generation criterion for $F$} \label{ss:p_generation}

In this section, we will outline a criterion, due to Golan \cite{GolanGeneration}, that guarantees that a finite subset of $F$ generates $F$. This criterion uses a construction known as the Stallings $2$-core, defined by Guba and Sapir \cite{gubaSapirDiagramGroups}, that arises when $F$ is viewed as a diagram group. However, since we are only concerned with $F$ and not diagram groups in general, we will give an outline of how to verify this criterion that does not explicitly involve diagram groups.  For a formal description of the algorithm see \cite{GolanGeneration} or \cite{Nikkel2019}.

Let $\gamma_1, \dots, \gamma_k \in F$. For each $1 \leq i \leq k$, let $A_i \to B_i$ be a reduced tree pair for $\gamma_i$. Let $\Gamma$ be the forest defined as the disjoint union of $A_1, B_1, \dots, A_k, B_k$, and let $V$ be set of vertices of $\Gamma$. Let us now define a directed graph $\Delta = \Delta(\gamma_1, \dots, \gamma_k)$ that captures the information of the Stallings $2$-core of $\{\gamma_1, \dots, \gamma_k\}$. We do this in several stages.

First let $\sim_0$ be the equivalence relation on $V$ defined such that for $u \neq v$ we have $u \sim_0 v$ if and only if $u$ and $v$ are both roots of $\Gamma$ or for some $i,j \in \N$ the vertices $u$ and $v$ are the $j$th leaves, with respect to the lexicographic order, of $A_i$ and $B_i$. 

Now define $\sim$ as the equivalence relation on $V$ obtained from $\sim_0$ by repeatedly carrying out the following operations (i) and (ii) until no more reductions can be made by repeating the operation. (Recall here that a triple $(v_0,r,v_1)$ of vertices of a binary rooted forest is a \emph{caret} if $v_0$ and $v_1$ are the left and right downward neighbours of $r$, respectively.)
\begin{enumerate}
\item For two carets $(u_0,r,u_1)$, $(v_0,s,v_1)$ in $\Gamma$, if $r \sim s$, then redefine $\sim$ as the finest coarsening of $\sim$ such that $u_0 \sim v_0$ and $u_1 \sim v_1$.
\item For two carets $(u_0,r,u_1), (v_0,s,v_1)$ in $\Gamma$, if $u_0 \sim v_0$ and $u_1 \sim v_1$, then redefine $\sim$ as the finest coarsening of $\sim$ such that $r \sim s$.
\end{enumerate}
By \cite[Remark~3.21]{GolanSapir}, $\sim$ does not depend on the order in which we carry out these operations. 

Finally, define $\Delta$ as the directed edge-labeled graph labeled by $\{0,1\}$ whose vertices are the equivalence classes of $\sim$ and where there is a directed edge with label $0$ (respectively, $1$) from $X$ to $Y$ if some $u \in X$ has a downward left (respectively, right) neighbour $v \in Y$. 

We are now in the position to state the following lemma, combining \cite[Corollary~1.4, Lemma~7.1 \& Remark~7.2]{GolanGeneration}, which gives a generation criterion for $F$. Here (and when we use the lemma in Proposition~\ref{prop:finite}), we denote the derivative of a function $\alpha$ by $\alpha'$. 

\begin{lemma} \label{lem:generation}
Let $\gamma_1, \dots, \gamma_k \in F$. Then $\< \gamma_1, \dots, \gamma_k \> = F$ if and only if all of the following hold:
\begin{enumerate}
\item there exists $\mu, \nu, \xi \in \< \gamma_1, \dots, \gamma_k\>$ such that 
\begin{enumerate}[{\rm (a)}]
\item $\mu'(0^+) = 2$ and $\mu'(1^-) = 1$
\item $\nu'(0^+) = 1$ and $\nu'(1^-) = 2$
\item there exists $x \in \Z[\frac{1}{2}] \cap (0,1)$ such that $\xi(x)=x$, $\xi'(x^+) = 2$ and $\xi'(x^-) = 1$. 
\end{enumerate}
\item $\Delta(\gamma_1,\dots,\gamma_k)$ is the directed graph in Figure~\ref{fig:generation}.
\end{enumerate}
\end{lemma}

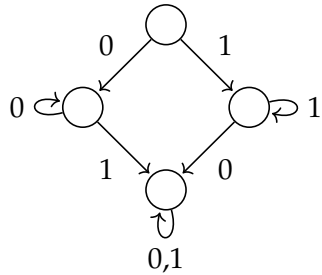
\begin{figure}[h]
\[
\begin{tikzpicture} 
[->, shorten >=1pt, auto, semithick, node distance=1cm, every state/.style={inner sep=1pt,minimum size=15pt}]
\node[state] (0)                      {};
\node[state] (1) [below left  = of 0] {};
\node[state] (2) [below right = of 0] {};
\node[state] (3) [below right = of 1] {};
\path (0) edge              node [swap] {0} (1)
          edge              node        {1} (2)
      (1) edge [loop left]  node        {0} (1)
          edge              node [swap] {1} (3)
      (2) edge              node        {0} (3)
          edge [loop right] node        {1} (2)
      (3) edge [loop below] node        {0,1} (3);
\end{tikzpicture}
\]
\caption{The directed graph for Lemma~\ref{lem:generation}}
\label{fig:generation}
\end{figure}

Let us justify why Lemma~\ref{lem:generation} follows from  \cite[Corollary~1.4, Lemma~7.1 \& Remark~7.2]{GolanGeneration}. This is a matter of translating notation. In terms of diagram groups, to obtain the diagram associated to the tree pair $A_i \to B_i$, we first take the graph obtained by identifying the leaves of $A_i$ with the corresponding leaves of $B_i$ and then we replace each vertex of $A_i$ and $B_i$ with an edge and each caret of $A_i$ and $B_i$ with a cell with top path that is a single edge that corresponds to the root of the caret and with bottom path that is a path of length two that corresponds to the two leaves of the caret. Under this correspondence, the algorithm that we described to repeatedly modify the equivalence relation $\sim$ corresponds to the algorithm given in \cite{GolanSapir} to repeatedly identify edges of the diagrams and the directed graph $\Delta$ that we obtain encodes the same information as the Stallings $2$-core defined in \cite[Section~3.2]{GolanSapir}.

We conclude by fixing a standard generating pair for $F$. Let $x_0 = (00, 01, 1) \mapsto (0, 10, 11)$ and $x_1 = (0, 100, 101, 11) \mapsto (0, 10, 110, 111)$, which we represent by tree pairs in Figure~\ref{fig:x} (we are using our notation from Section~\ref{ss:p_notation}). Then, by \cite[Theorem~3.4]{cfp} for example, $F = \< x_0, x_1 \>$. (It is an instructive exercise to verify this also using Lemma~\ref{lem:generation} above.) Moreover, although we do not need this fact, if we inductively define $x_{i+1} = x_i^{x_0}$ for all $i \geq 1$, then the elements $x_0, x_1, x_2, \dots$ witness the following well known presentation $F = \< x_0, x_1, x_2, \dots \mid \text{$x_j^{x_i} = x_{j+1}$ for $i < j$} \>$.

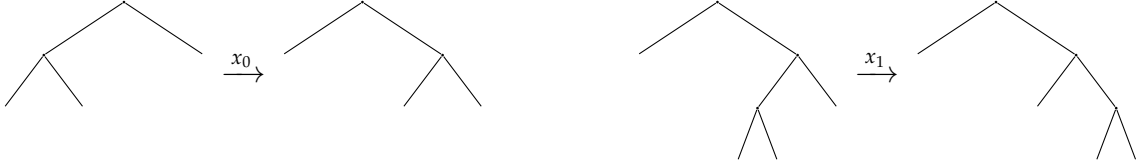
\begin{figure}[h]
\begin{gather*}
  \begin{tikzpicture}[
      inner sep=0pt,
      baseline=-30pt,
      level distance=20pt,
      level 1/.style={sibling distance=60pt},
      level 2/.style={sibling distance=30pt},
      level 3/.style={sibling distance=15pt}
    ]
    \node (root) [circle,fill] {}
    child {node (0) [circle,fill] {}
      child {node (00) {}}
      child {node (01) {}}}
    child {node (1) {}};
  \end{tikzpicture}
  \ \xrightarrow{\,x_0\,} \
  \begin{tikzpicture}[
      inner sep=0pt,
      baseline=-30pt,
      level distance=20pt,
      level 1/.style={sibling distance=60pt},
      level 2/.style={sibling distance=30pt},
      level 3/.style={sibling distance=15pt}
    ]
    \node (root) [circle,fill] {}
    child {node (0) {}}
    child {node (1) [circle,fill] {}
      child {node (10) {}}
      child {node (11) {}}};
  \end{tikzpicture}
  \hspace{2cm}
  \begin{tikzpicture}[
      inner sep=0pt,
      baseline=-30pt,
      level distance=20pt,
      level 1/.style={sibling distance=60pt},
      level 2/.style={sibling distance=30pt},
      level 3/.style={sibling distance=15pt}
    ]
    \node (root) [circle,fill] {}
    child {node (0) {}}
    child {node (1) [circle,fill] {}
      child {node (10) [circle,fill] {}
        child {node (100) {}}
        child {node (101) {}}}
      child {node (11) {}}};       
  \end{tikzpicture} 
  \ \xrightarrow{\,x_1\,} \
  \begin{tikzpicture}[
      inner sep=0pt,
      baseline=-30pt,
      level distance=20pt,
      level 1/.style={sibling distance=60pt},
      level 2/.style={sibling distance=30pt},
      level 3/.style={sibling distance=15pt}
    ]
    \node (root) [circle,fill] {}
    child {node (0) {}}
    child {node (1) [circle,fill] {}
      child {node (10) {}}
      child {node (11) [circle,fill] {}
        child {node (110) {}}
        child {node (111) {}}}};
  \end{tikzpicture}
\end{gather*}
\caption{The elements $x_0, x_1 \in F$} 
\label{fig:x}
\end{figure}

\section{Proof of Theorem \ref{t:theorem}} \label{s:proofs}

We begin with the following result that handles elements of infinite order.

\begin{proposition} \label{prop:infinite}
Let $\alpha, \zeta \in T$ have infinite order. Then there exists $\gamma \in T$ such that $T = \< \alpha, \zeta^\gamma \>$.
\end{proposition}

\begin{proof}
We begin by identifying some properties of the element $\alpha$. By \cite[Corollary~J]{GhysSergiescu}, every element of $T$ has a finite orbit on $\Sone$, so $\alpha$ has a finite orbit, and by replacing $\alpha$ by a power if necessary, we assume that $\fix{\alpha} \neq \emptyset$.  Let $a\in\supt{\alpha}\cap\Z[\frac{1}{2}]$. There is an interval $(c,d)\subset \Sone$ with $c<a<d\leq c$ which is the (open) component of support of $\alpha$ containing the point $a$.  By replacing $\alpha$ by its inverse if necessary we may assume that $a<a\alpha<d$ in the circular order and in general that $\alpha$ is increasing over $(c,d)$. Writing $a_i = a\alpha^i$ for each $i \in \Z$, we have $[a_i,a_{i+1}) \cap [a_j,a_{j+1}) = \emptyset$ for distinct $i,j \in \Z$. In particular, $\alpha$ admits ${18}$ hops at $a$, with fundamental domains $[a_0,a_1), \dots, {[a_{17},a_{18})}$. 

We have shown these fundamental domains on Figure~\ref{fig:infinite}. Many other labels are also given in this figure, and we will introduce these labels later in the proof (working outwards from the centre of the circle), so the reader is encouraged to consult this figure throughout the proof.

We will first identify an element $\gamma \in T$ such that $(T_{[a_8, a_{10}]})' \leq \< \alpha, \zeta^\gamma \>$. To this end, we need to fix a generating set for $T_{[a_8, a_{10}]}$. Fix $x_0, x_1 \in F$ from Figure~\ref{fig:x}, so $F = \< x_0, x_1 \>$. Fix a Thompson-like map $\tau\: (0,1) \to (a_8, a_{10})$ such that $\frac{1}{2}\tau = a_9$. By Lemma~\ref{lem:interval_t}, $T_{[a_8, a_{10}]} = F^\tau = \< x_0^\tau, x_1^\tau \>$. 

Motivated by this, define $\beta = \alpha x_0^{\tau\alpha^{-2}} x_1^{\tau}$. Note that $x_0^{\tau\alpha^{-2}}$ stabilises $[a_6, a_8]$ and acts trivially on the complement. Similarly $x_1^{\tau}$ stabilises $[a_8, a_{10}]$. Write $b_i = a\beta^i$ for each $i \in \Z$.  We observe that $\beta$ also admits ${18}$ hops at $a$, with fundamental domains $[b_0,b_1), \dots, {[b_{17},b_{18})}$. Note that $\beta|_{\Sone \setminus [a_5,a_9]} = \alpha|_{\Sone \setminus [a_5,a_9]}$, so $b_i = a_i$ for all $0 \leq i \leq 6$. 

We claim that $a_i < b_i < a_{i+1}$ for all $7 \leq i \leq 17$. To see this, we begin with $i=7$, where 
\[
b_7 = b_6 \beta = b_6 \alpha x_0^{\tau\alpha^{-2}} x_1^{\tau} = a_6 \alpha x_0^{\tau\alpha^{-2}} x_1^{\tau} = a_7  x_0^{\tau\alpha^{-2}} x_1^{\tau} \in (a_7,a_8)
\]
since $x_0^{\tau\alpha^{-2}}$ is increasing on $(a_6,a_8)$. Next let $8 \leq i \leq 9$. Here $b_{i-1} \in (a_{i-1},a_i)$, so $b_{i-1} \alpha \in (a_i,a_{i+1})$. Now $x_1$ stabilises $(0,\frac{1}{2})$ and $(\frac{1}{2},1)$, and $(0,\frac{1}{2})\tau = (a_8,a_9)$ and $(\frac{1}{2},1)\tau = (a_9,a_{10})$ by construction, so $x_1^\tau$ stabilises $(a_8,a_9)$ and $(a_9,a_{10})$. Therefore $b_i = b_{i-1} \alpha x_0^{\tau\alpha^{-2}} x_1^{\tau} \in (a_i,a_{i+1})$. Finally, consider $10 \leq i \leq 17$. Here $b_{i-1}\alpha \in (a_i,a_{i+1})$ as $b_{i-1} \in (a_{i-1},a_i)$, so $b_i = b_{i-1}\alpha$ as $(a_i,a_{i+1}) \subseteq \Sone \setminus [a_5,a_9]$.

We will soon choose $\gamma \in T$ such that $\zeta^\gamma$ agrees with $\beta$ on $[b_0,b_{15}]$, but we want to have some control over $\zeta^\gamma$ in the region $\Sone \setminus [b_0,b_{15}]$ too, which necessitates the following discussion. Record that, since $c \in \fix{\alpha}$, we have
\begin{equation} \label{eq:alpha}
\text{$y \in (b_{15},c) \implies y\alpha \in (b_{16},c)$\quad and \quad $y \in (c,a_1) \implies y\alpha^{-1} \in (c,a_0)$.}
\end{equation}

Now fix $r,s \in \Z[\frac{1}{2}]$ such that $b_{16} < r < b_{17}$ and $a_{-1} < s < a_0$. Since $\zeta$ has infinite order, by Lemma~\ref{lem:conjugacy}, there exists $\gamma \in T$ such that $\zeta^\gamma|_{[b_0,b_{15}]} = \beta|_{[b_0,b_{15}]}$ and, in the circular order,
\[
b_{16} < r < b_{17} < c < a_{-1} < s < r\zeta^\gamma < s\zeta^\gamma < a_0 < b_{16}.
\]  
In particular, note that for all nonnegative integers $i$ we have 
\begin{equation} \label{eq:zeta}
\text{$b_{17}\alpha^i\zeta^\gamma \in (c,a_0)$\quad and  \quad $a_{-1}\alpha^{-i}(\zeta^\gamma)^{-1} \in (b_{15},r)\subsetneq (b_{15},c)$.}
\end{equation}

Having defined $\gamma \in T$ we will now prove that $(T_{[a_8,a_{10}]})' \leq \< \alpha, \zeta^\gamma \>$. To this end, let $\eta = \alpha^{-1}\zeta^\gamma$. Since $\zeta^\gamma|_{[b_0,b_{15}]} = \beta|_{[b_0,b_{15}]}$ and $\beta|_{\Sone \setminus [a_5,a_9]} = \alpha|_{\Sone \setminus [a_5,a_9]}$, we deduce that $\eta|_{[a_6,a_{10}]} = x_0^{\tau\alpha^{-2}} x_1^{\tau}$ and $\eta$ acts trivially on $[a_1,a_6] \cup [a_{10},b_{16}]$, so $\supt{\eta} \subseteq (a_6,a_{10}) \cup (b_{16},a_1)$. In particular, $\eta|_{[a_8,a_{10}]} = x_1^\tau$ and $\eta^{\alpha^2}|_{[a_8,a_{10}]} = x_0^\tau$, but we have no control over  $\eta$ in the region $(b_{16},a_1)$ nor over $\eta^{\alpha^2}$ in the region $(b_{18},a_3)$, so we will replace these elements with suitable conjugates to solve this problem.

Let $\mu_0 = (\eta^{\alpha^2})^{\alpha^2\zeta^\gamma\alpha^{-3}}$ and $\mu_1 = \eta^{\alpha^{-4}(\zeta^\gamma)^{-1}\alpha^5}$, noting $\mu_0|_{[a_8,a_{10}]} = x_0^\tau$ and $\mu_1|_{[a_8,a_{10}]} = x_1^\tau$. We claim
\[
\text{$\supt{\mu_0} \subseteq (a_8,a_{12}) \cup (c,a_3)$\quad and \quad $\supt{\mu_1} \subseteq (a_6,a_{10}) \cup (b_{16},c)$.} 
\]

First consider $\supt{\mu_0}$. Using the fact that $\zeta^{\gamma}|_{[a_0,a_5] \cup [a_9,b_{15}]} = \beta|_{[a_0,a_5] \cup [a_9,b_{15}]} = \alpha|_{[a_0,a_5] \cup [a_9,b_{15}]}$, we see that $(a_6,a_{10})\alpha^2(\alpha^2\zeta^\gamma\alpha^{-3}) = (a_8,a_{12})$ and $a_1\alpha^2(\alpha^2\zeta^\gamma\alpha^{-3}) = a_3$. Now by applying \eqref{eq:zeta} then \eqref{eq:alpha}, we obtain $(b_{16}\alpha^2)\alpha^2\zeta^\gamma\alpha^{-3} \in (c,a_0)$, which implies that $(b_{16},a_1)\alpha^2(\alpha^2\zeta^\gamma\alpha^{-3}) \subseteq (c,a_3)$. Similarly for $\supt{\mu_1}$, we check that $(a_6,a_{10})\alpha^{-4}(\zeta^\gamma)^{-1}\alpha^5 = (a_6,a_{10})$ and $b_{16}\alpha^{-4}(\zeta^\gamma)^{-1}\alpha^5 = b_{16}$, and by applying \eqref{eq:zeta} then \eqref{eq:alpha} we see that $a_1\alpha^{-4}(\zeta^\gamma)^{-1}\alpha^5 \in (b_{16},c)$, which implies that $(b_{16},a_1)\alpha^{-4}(\zeta^\gamma)^{-1}\alpha^5 \subseteq (b_{16},c)$. This proves the claim.

In particular, $\supt{\mu_0} \cap \supt{\mu_1} \subseteq (a_8,a_{10})$, so 
\[
\< \mu_0, \mu_1 \>' = \< \mu_0|_{[a_8,a_{10}]}, \mu_1|_{[a_8,a_{10}]} \>' = \< x_0^{\tau\alpha^2}, x_1^{\tau\alpha^2} \>' = (T_{[a_8,a_{10}]})'.
\] 

Now that we know $(T_{[a_8,a_{10}]})' \leq \< \alpha, \zeta^\gamma \>$, we next seek to prove that $T = \< \alpha, \zeta^\gamma\>$. First fix $p,q \in \Z[\frac{1}{2}]$ such that $a_8 < p < a_9 < p\alpha < q < a_{10}$ and note that Lemma~\ref{lem:simple_f} implies that $T_{[p,q]} \leq (T_{[a_8,a_{10}]})' \leq \< \alpha, \zeta^\gamma \>$. Moreover, by the choice of $p$ and $q$, we see that $(p,q)\alpha^i \cap (p,q)\alpha^{i+1} \neq \emptyset$. 

We seek to apply Lemma~\ref{lem:small_support}. We claim that $\Sone = \bigcup_{\delta \in \Delta} (p,q)\delta$ where
\[
\Delta = \{ \alpha^i \mid -10 \leq i \leq 9 \} \cup \{ \alpha^{-10}(\zeta^{\gamma})^{-1}, \alpha^{-9}(\zeta^{\gamma})^{-1} \}.
\]
To see this, make the following two observations:
\begin{gather*}
\bigcup_{-10 \leq i \leq 9} (p,q)\alpha^i = (p\alpha^{-10},q\alpha^9) \supseteq [a_{-1},a_{18}] \supsetneq \Sone \setminus (r,s) \\
(p,q)\alpha^{-10}(\zeta^{\gamma})^{-1} \cup (p,q)\alpha^{-9}(\zeta^{\gamma})^{-1} = (p\alpha^{-10},q\alpha^{-9})(\zeta^{\gamma})^{-1} \supseteq [a_{-1},a_0](\zeta^{\gamma})^{-1} \supsetneq (r,s).
\end{gather*}
Therefore, by Lemma~\ref{lem:small_support}, $T = \< \{ T_{[p,q]}^\delta : \delta \in \Delta \} \> \leq \< T_{[p,q]}, \Delta \> \leq \< \alpha, \zeta^\gamma \>$. 
\end{proof}

\begin{figure}
\[
\begin{tikzpicture}[scale=0.8]
\draw[black] (0,0) circle (5);

\draw[black] (90:5) -- (90:5.5);
\draw (90:5.8) node[font=\large] {$c$};

\foreach \i [count=\j from 0] in {180,192,...,396} {
    \draw[black!30] (\i:5) -- (\i:4.5);
    \draw (\i:4.2) node[font=\large,black!30] {$a_{\j}$};
}
\draw[black] (168:5) -- (168:4.5);
\draw (168:4.2) node[font=\large] {\,$a_{-1}$};
\draw[black] (180:5) -- (180:4.5);
\draw (180:4.2) node[font=\large] {$a_0$};
\draw[black] (252:5) -- (252:4.5);
\draw (252:4.2) node[font=\large] {$a_6$};
\draw[black] (276:5) -- (276:4.5);
\draw (276:4.2) node[font=\large] {$a_8$};
\draw[black] (300:5) -- (300:4.5);
\draw (300:4.2) node[font=\large] {$a_{10}$};
\foreach \i [count=\j from 0] in {180,192,...,252} {
    \draw[black!30] (\i:5) -- (\i:5.5);
    \draw (\i:5.8) node[font=\large,black!30] {$b_{\j}$};
}

\foreach \i [count=\j from 7] in {270,282,...,366} {
    \draw[black!30] (\i:5) -- (\i:5.5);
    \draw (\i:5.8) node[font=\large,black!30] {$b_{\j}$};
}

\draw[black] (378:5) -- (378:5.5);
\draw (378:5.8) node[font=\large] {$b_{16}$};
\draw[black] (390:5) -- (390:5.5);
\draw (390:5.8) node[font=\large] {$b_{17}$};

\draw[black] (387:5) -- (387:6.25);
\draw (387:6.55) node[font=\large] {$r$};
\draw[black] (171:5) -- (171:6.25);
\draw (171:6.55) node[font=\large] {\raisebox{1.5mm}{$s$}};
\draw[black] (174:5) -- (174:6.25);
\draw (174:6.55) node[font=\large] {\raisebox{1mm}{$r\zeta^\gamma$}};
\draw[black] (177:5) -- (177:6.25);
\draw (177:6.55) node[font=\large] {\raisebox{-6mm}{$s\zeta^\gamma$}};

\foreach \i in {159,171,...,387} {
  \draw[draw=white,fill=black!30,opacity=0.5,rounded corners=2mm] (\i+18:7) arc (\i+18:\i:7) -- (\i:7.5) arc (\i:\i+18:7.5) -- cycle;
}

\draw[draw=black,rounded corners=2mm] (297:7) arc (297:279:7) -- (279:7.5) arc (279:297:7.5) -- cycle;
\draw[black] (279:5) -- (279:7.7);
\draw (279:8) node[font=\large,black] {$p$};
\draw[black] (297:5) -- (297:7.7);
\draw (297:8) node[font=\large,black] {$q$};

\draw[draw=black,rounded corners=2mm] (177:7) arc (177:159:7) -- (159:7.5) arc (159:177:7.5) -- cycle;
\draw (159:8) node[font=\large,black] {$p\alpha^{-10}$};
\draw (177:8) node[font=\large,black] {$q\alpha^{-10}\phantom{-}$};

\draw[draw=black,rounded corners=2mm] (405:7) arc (405:387:7) -- (387:7.5) arc (387:405:7.5) -- cycle;
\draw (387:8) node[font=\large,black] {$p\alpha^9$};
\draw (405:8) node[font=\large,black] {$q\alpha^9$};
\end{tikzpicture}
\]
\caption{An accompaniment to the proof of Proposition~\ref{prop:infinite}} \label{fig:infinite}
\end{figure}

We now turn to elements of finite order. Here our calculations are more detailed and we describe elements of $T$ with notation from Section~\ref{ss:p_notation}. For what remains, fix $\zeta \in T$ as the specific element $(00, 01, 10, 11) \mapsto (0, 100, 101, 11)$, which is given as a tree pair in Figure~\ref{fig:zeta}.

\begin{proposition} \label{prop:finite}
Let $1 \neq \alpha \in T$ have finite order. Then there exists $\gamma \in T$ such that $T = \< \alpha, \zeta^\gamma \>$.
\end{proposition}

\begin{proof}
By replacing $\alpha$ by a power if necessary, we assume that $\alpha$ has rotation number $\frac{1}{p}$ for a prime $p$. In the terminology of Belk and Matucci in \cite{BelkMatucci14}, any two finite order elements in $T$ with the same rotation number $\frac{p}{q}$ (in reduced terms) have their reduced strand diagrams both consisting of a single strand which wraps $p$ times horizontally along the torus while wrapping $q$ times vertically, so \cite[Theorem~2.12]{BelkMatucci14} implies that they are conjugate elements of $T$. In particular, we can replace $\alpha$ by a conjugate if necessary, so we will assume that $\alpha$ is one of:
\begin{enumerate}[(a)]
\item $(0, 1) \mapsto (1, 0)$
\item $(0, 10, 11) \mapsto (10, 11, 0)$
\item $(00, 01, 10, 110, \dots, 1^{p-3}0, 1^{p-2}) \mapsto (01, 10, 110, \dots, 1^{p-3}0, 1^{p-2}, 00)$ for some prime $p \geq 5$.
\end{enumerate}
In all three cases, we will prove that $T = \< \alpha, \zeta \>$.
\vspace{5pt}

\emph{\textbf{Case~(a).} $\alpha = (0, 1) \mapsto (1, 0)$.}\nopagebreak

Let $\kappa_0 = \zeta$ and $\kappa_1 = \zeta^{(\alpha\zeta)^2}$, which are the maps
\begin{gather*}
\kappa_0 = (00,01,10,11) \mapsto (0,100,101,11) \\
\kappa_1 = (00,010,011,100,101,11) \mapsto (00,01,100,101,110,111).
\end{gather*}
We will first show that $\<\kappa_0,\kappa_1\> = T_{[0,1]} = F$.

To do this, we will apply the generation criterion of Golan as given in Lemma~\ref{lem:generation}. In particular, we need to show that conditions~(i) and~(ii) of Lemma~\ref{lem:generation} are satisfied. For condition~(i), it suffices to check that $\mu = \kappa_0$, $\nu = \kappa_1^{-1}$, $\xi = \kappa_1$ and $x = \frac{1}{4}$ satisfy the required conditions. We now turn to condition~(ii), for which we must compute the directed graph $\Delta = \Delta(\kappa_0,\kappa_1)$. We follow the procedure in Section~\ref{ss:p_generation}. First note that the tree pairs for $\kappa_0$ and $\kappa_1$ corresponding to the expressions given above are reduced. Next we compute the equivalence relation $\sim_0$ on the vertex set of all four trees in the tree pairs, which we present in Figure~\ref{fig:finite_a_sim_0} by labelling vertices according to their equivalence class. Now we repeatedly modify the equivalence relation as described in Section~\ref{ss:p_generation} to obtain the equivalence relation $\sim$ shown in Figure~\ref{fig:finite_a_sim}. Finally, this yields the directed graph $\Delta$ in Figure~\ref{fig:finite_a_delta}, which matches the graph in Figure~\ref{fig:generation}, as required. Therefore, Lemma~\ref{lem:generation} implies that $F = \< \kappa_0, \kappa_1 \>$. 

In particular, $T_{[0,1]} = F \leq \< \alpha, \zeta \>$. Since $0\alpha = \frac{1}{2}$, we have $(0,1)\alpha = (\frac{1}{2},\frac{3}{2})$, so $\Sone = (0,1) \cup (0,1)^\alpha$. Now Lemma~\ref{lem:small_support} implies that $T = \< T_{[0,1]}, (T_{[0,1]})^\alpha \> = \< \alpha, \zeta \>$.

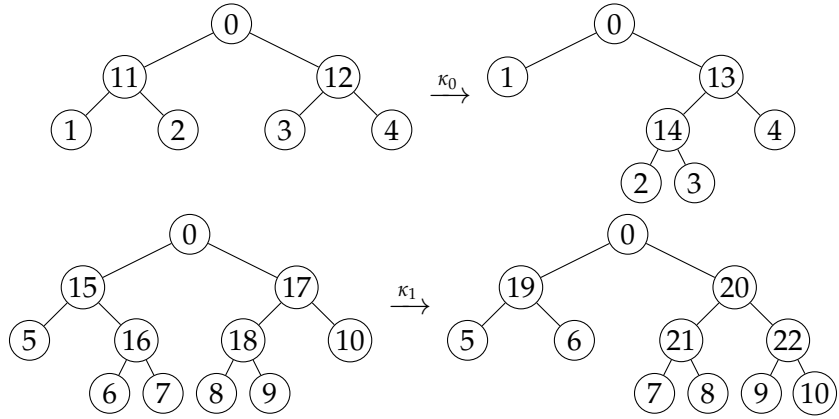
\begin{figure}[ht]
\begin{gather*}
\begin{tikzpicture}
    [ inner sep=1pt,
      minimum size=15pt,
      baseline=-30pt,
      level distance=20pt,
      level 1/.style={sibling distance=80pt},
      level 2/.style={sibling distance=40pt},
      level 3/.style={sibling distance=20pt}
    ]
    \node (root) [circle, draw] {0}
    child {node (0) [circle, draw] {11}
      child {node (00)  [circle, draw] {1}}
      child {node (01) [circle, draw] {2}}}
    child {node (1) [circle, draw] {12}
      child {node (10) [circle, draw] {3}}
      child {node (11) [circle, draw] {4}}};
\end{tikzpicture}
\ \xrightarrow{\,\kappa_0\,} \
\begin{tikzpicture}
    [ inner sep=1pt,
      minimum size=15pt,
      baseline=-30pt,
      level distance=20pt,
      level 1/.style={sibling distance=80pt},
      level 2/.style={sibling distance=40pt},
      level 3/.style={sibling distance=20pt}
    ]
    \node (root) [circle, draw] {0}
    child {node (0) [circle, draw] {1}}
    child {node (1) [circle, draw] {13}
      child {node (10) [circle, draw] {14}
        child {node (100) [circle, draw] {2}}
        child {node (101) [circle, draw] {3}}}
      child {node (11) [circle, draw] {4}}};
\end{tikzpicture} \\
\begin{tikzpicture}
    [ inner sep=1pt,
      minimum size=15pt,
      baseline=-30pt,
      level distance=20pt,
      level 1/.style={sibling distance=80pt},
      level 2/.style={sibling distance=40pt},
      level 3/.style={sibling distance=20pt}
    ]
    \node (root) [circle, draw] {0}
    child {node (0) [circle, draw] {15}
      child {node (00) [circle, draw] {5}}
      child {node (01) [circle, draw] {16}
        child {node (010) [circle, draw] {6}}
        child {node (011) [circle, draw] {7}}}}
    child {node (1) [circle, draw] {17}
      child {node (10) [circle, draw] {18}
        child {node (100) [circle, draw] {8}}
        child {node (101) [circle, draw] {9}}}
      child {node (11) [circle, draw] {10}}};
\end{tikzpicture}
\ \xrightarrow{\,\kappa_1\,} \
\begin{tikzpicture}
    [ inner sep=1pt,
      minimum size=15pt,
      baseline=-30pt,
      level distance=20pt,
      level 1/.style={sibling distance=80pt},
      level 2/.style={sibling distance=40pt},
      level 3/.style={sibling distance=20pt}
    ]
    \node (root) [circle, draw] {0}
    child {node (0) [circle, draw] {19}
      child {node (00) [circle, draw] {5}}
      child {node (01) [circle, draw] {6}}}
    child {node (1) [circle, draw] {20}
      child {node (10) [circle, draw] {21}
        child {node (100) [circle, draw] {7}}
        child {node (101) [circle, draw] {8}}}
      child {node (11) [circle, draw] {22}
        child {node (110) [circle, draw] {9}}
        child {node (111) [circle, draw] {10}}}};
\end{tikzpicture}
\end{gather*}
\caption{The initial relation $\sim_0$ for $\{ \kappa_0, \kappa_1 \}$ in Case~(a) in the proof of Proposition~\ref{prop:finite}}
\label{fig:finite_a_sim_0}
\end{figure}

\begin{figure}[h]
\begin{gather*}
\begin{tikzpicture}
    [ inner sep=1pt,
      minimum size=15pt,
      baseline=-30pt,
      level distance=20pt,
      level 1/.style={sibling distance=80pt},
      level 2/.style={sibling distance=40pt},
      level 3/.style={sibling distance=20pt}
    ]
    \node (root) [circle, draw] {0}
    child {node (0) [circle, draw] {1}
      child {node (00)  [circle, draw] {1}}
      child {node (01) [circle, draw] {2}}}
    child {node (1) [circle, draw] {4}
      child {node (10) [circle, draw] {2}}
      child {node (11) [circle, draw] {4}}};
\end{tikzpicture}
\ \xrightarrow{\,\kappa_0\,} \
\begin{tikzpicture}
    [ inner sep=1pt,
      minimum size=15pt,
      baseline=-30pt,
      level distance=20pt,
      level 1/.style={sibling distance=80pt},
      level 2/.style={sibling distance=40pt},
      level 3/.style={sibling distance=20pt}
    ]
    \node (root) [circle, draw] {0}
    child {node (0) [circle, draw] {1}}
    child {node (1) [circle, draw] {4}
      child {node (10) [circle, draw] {2}
        child {node (100) [circle, draw] {2}}
        child {node (101) [circle, draw] {2}}}
      child {node (11) [circle, draw] {4}}};
\end{tikzpicture} \\
\begin{tikzpicture}
    [ inner sep=1pt,
      minimum size=15pt,
      baseline=-30pt,
      level distance=20pt,
      level 1/.style={sibling distance=80pt},
      level 2/.style={sibling distance=40pt},
      level 3/.style={sibling distance=20pt}
    ]
    \node (root) [circle, draw] {0}
    child {node (0) [circle, draw] {1}
      child {node (00) [circle, draw] {1}}
      child {node (01) [circle, draw] {2}
        child {node (010) [circle, draw] {2}}
        child {node (011) [circle, draw] {2}}}}
    child {node (1) [circle, draw] {4}
      child {node (10) [circle, draw] {2}
        child {node (100) [circle, draw] {2}}
        child {node (101) [circle, draw] {2}}}
      child {node (11) [circle, draw] {4}}};
\end{tikzpicture}
\ \xrightarrow{\,\kappa_1\,} \
\begin{tikzpicture}
    [ inner sep=1pt,
      minimum size=15pt,
      baseline=-30pt,
      level distance=20pt,
      level 1/.style={sibling distance=80pt},
      level 2/.style={sibling distance=40pt},
      level 3/.style={sibling distance=20pt}
    ]
    \node (root) [circle, draw] {0}
    child {node (0) [circle, draw] {1}
      child {node (00) [circle, draw] {1}}
      child {node (01) [circle, draw] {2}}}
    child {node (1) [circle, draw] {4}
      child {node (10) [circle, draw] {2}
        child {node (100) [circle, draw] {2}}
        child {node (101) [circle, draw] {2}}}
      child {node (11) [circle, draw] {4}
        child {node (110) [circle, draw] {2}}
        child {node (111) [circle, draw] {4}}}};
\end{tikzpicture}
\end{gather*}
\caption{The final relation $\sim$ for $\{ \kappa_0, \kappa_1 \}$ in Case~(a) in the proof of Proposition~\ref{prop:finite}}
\label{fig:finite_a_sim}
\end{figure}

\begin{figure}[ht]
\centering
\begin{tikzpicture} 
[scale=0.9, ->, >=stealth', shorten >=1pt, auto, semithick, node distance=1cm, every state/.style={inner sep=1pt,minimum size=15pt}]
\node[state] (0)                      {0};
\node[state] (1) [below left  = of 0] {1};
\node[state] (2) [below right = of 0] {4};
\node[state] (3) [below right = of 1] {2};
\path (0) edge              node [swap] {0} (1)
          edge              node        {1} (2)
      (1) edge [loop left]  node        {0} (1)
          edge              node [swap] {1} (3)
      (2) edge              node        {0} (3)
          edge [loop right] node        {1} (2)
      (3) edge [loop below] node        {0,1} (3);
\end{tikzpicture}
{\setlength{\abovecaptionskip}{2pt}
\caption{The directed graph $\Delta(\kappa_0, \kappa_1)$ in Case~(a) in the proof of Proposition~\ref{prop:finite}}
\label{fig:finite_a_delta}
}
\end{figure}
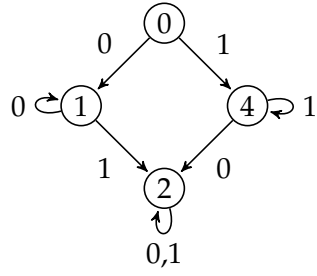

\emph{\textbf{Case~(b).}  $\alpha = (0, 10, 11) \mapsto (10, 11, 0)$.}\nopagebreak

This case is straightforward. Observe that
\begin{gather*}
\zeta^\alpha = (0, 100, 101, 11) \mapsto (0, 10, 110, 111) = x_1  \\
\zeta\zeta^\alpha = (00, 01, 10, 11) \mapsto (0, 10, 110, 111) = (00, 01, 1) \mapsto (0, 10, 11) = x_0.
\end{gather*} 
Therefore, $T_{[0,1]} = F = \< x_0, x_1 \> \leq \< \alpha, \zeta \>$. As in Case~(a), $0\alpha = \frac{1}{2}$, so $\Sone = (0,1) \cup (0,1)^\alpha$, and Lemma~\ref{lem:small_support} implies that $T = \< T_{[0,1]}, (T_{[0,1]})^\alpha \> = \< \alpha, \zeta \>$. 
\vspace{5pt}

\emph{\textbf{Case~(c).} $\alpha = (00, 01, 10, 110, \dots, 1^{p-3}0, 1^{p-2}) \mapsto (01, 10, 110, \dots, 1^{p-3}0, 1^{p-2}, 00)$ for $p \geq 5$.} \nopagebreak

Let $\kappa_0 = \zeta$ and $\kappa_1 = \zeta^\alpha$, which are the maps
\begin{gather*}
\kappa_0 = (00,01,10,11) \mapsto (0,100,101,11) \\
\kappa_1 = (00,010,011,10,110,111) \mapsto (00,01,10,1100,1101,111).
\end{gather*}
 We will first show that $\<\kappa_0,\kappa_1\> = T_{[0,\frac{7}{8}]}$. Let $\tau\:(0,\frac{7}{8}) \to (0,1)$ be defined for $x \in (0,\frac{7}{8})$ as 
\[
x\tau = \left\{ 
\begin{array}{ll}
x              & \text{if $x \leq \frac{3}{4}$} \\[3pt]
2x-\frac{3}{4} & \text{if $x > \frac{3}{4}$.}
\end{array}
\right.
\]
Thus, to prove $\< \kappa_0, \kappa_1 \> = T_{[0,\frac{7}{8}]}$, we will prove $\< \kappa_0^\tau, \kappa_1^\tau \> = (T_{[0,\frac{7}{8}]})^\tau = F$. Since $\supt{\kappa_0} = (0,\frac{3}{4})$, we have $\kappa_0^\tau = \kappa_0$, and it is easy to verify that $\kappa_1^\tau = (00,010,011,1) \mapsto (00,01,10,11)$. We will prove $\< \kappa_0^\tau, \kappa_1^\tau \> = (T_{[0,\frac{7}{8}]})^\tau = F$ by applying Lemma~\ref{lem:generation}. Condition~(i) is verified by choosing $\mu = \kappa_0^\tau$, $\nu = (\kappa_1^\tau)^{-1}$, $\xi = \kappa_1^\tau$ and $x = \frac{1}{4}$, and condition~(ii) is verified by computing $\Delta(\kappa_0^\tau,\kappa_1^\tau)$ in the same way as in Case~(a) and noting that it is the directed graph in Figure~\ref{fig:generation}. Therefore, Lemma~\ref{lem:generation} implies that $F = \< \kappa_0^\tau, \kappa_1^\tau \>$, which in turn implies that $T_{[0,\frac{7}{8}]} = \<\kappa_0, \kappa_1 \> \leq \< \alpha, \zeta \>$. Noting that $\Sone = \bigcup_{i \in \Z} (0,\frac{7}{8})\alpha^i$, by Lemma~\ref{lem:small_support}, we conclude that $\< \alpha, \zeta \> = \< T_{[0,\frac{7}{8}]}, \alpha \> = T$. 
\end{proof}

Theorem~\ref{t:theorem} now follows by combining Propositions~\ref{prop:infinite} and~\ref{prop:finite}.

\begin{acknowledgements}
The work in this paper emerged from the Focused Research Workshop \emph{Generating Thompson Groups} sponsored by the Heilbronn Institute for Mathematical Research. The third author was partially supported by the NSF grant DMS--2005297. We are also grateful to Matthew G. Brin and the anonymous referees for useful comments.
\end{acknowledgements}

\vspace{11pt}

\noindent Collin Bleak, \texttt{collin.bleak@st-andrews.ac.uk} \newline
University of St Andrews, St Andrews, Scotland, UK, KY16 9SS

\noindent Scott Harper, \texttt{scott.harper@st-andrews.ac.uk} \newline
University of Bristol \& Heilbronn Institute for Mathematical Research, Bristol, UK, BS8 1UG  \newline
\emph{current institution:} University of St Andrews, St Andrews, Scotland, UK, KY16 9SS

\noindent Rachel Skipper, \texttt{skipper.rachel.k@gmail.com} \newline
The Ohio State University, Columbus, Ohio, USA, 43210 \newline
\emph{current institution:} \'Ecole normale sup\'erieure, Paris, France, 75230 CEDEX 05
\end{document}